\documentclass[letter,11pt]{article}
\usepackage{amssymb}
\usepackage{amsfonts}
\usepackage{amsmath}
\usepackage{amsxtra}
\usepackage{amsthm}
\usepackage{mathrsfs}
\usepackage{enumerate}
\usepackage{geometry} % to change the page dimensions
\usepackage{graphicx} % support the \includegraphics command and options

\newtheorem{thm}{Theorem}[section]
\newtheorem{lem}[thm]{Lemma}
\newtheorem{cor}[thm]{Corollary}

\renewcommand{\epsilon}{\varepsilon}

\title{On higher order extensions for the fractional Laplacian}
\author{Ray Yang\\
  Courant Institute of Mathematical Sciences\\
  New York University\\
  251 Mercer Street \\
  New York, New York 10012, USA\\
  \texttt{ryang@cims.nyu.edu}}
\begin{document}
\maketitle
\begin{abstract}
The technique of Caffarelli and Silvestre \cite{C-S}, characterizing the fractional Laplacian as the Dirichlet-to-Neumann map for a function $U$ satisfying an elliptic equation in the upper half space with one extra spatial dimension, is shown to hold for general positive, non-integer orders of the fractional Laplace operator, by showing an equivalence between the $H^s$ norm on the boundary and a suitable higher-order seminorm of $U$.
\end{abstract}

\section{Introduction}
In \cite{C-S}, Caffarelli and Silvestre showed that, for $0 < \gamma < 1$, the fractional Laplacian $(-\Delta)^\gamma$ of a function $f$ living on $\mathbb{R}^n$ can be understood as the Dirichlet-to-Neumann map for a function $U$ living on the upper half-space $\mathbb{R}^{n+1}_+$, where $U$ coincided with $f$ on $\mathbb{R}^n$, and $U$ satisfied a particular 2nd-order elliptic equation. The function $U$ was thus a suitable extension of $f$. The extension was proved in two separate ways: firstly through an analysis of the Poisson kernel, and secondly through an energy equality where the $H^\gamma$ seminorm of $f$ was shown to be equivalent to a suitable Dirichlet-like energy for $U$, taking the form $\int_{\mathbb{R}^{n+1}_+} y^{1-2\gamma} |\nabla U|^2 dx dy$.  

In this note, we generalize the energy equality of Caffarelli and Silvestre to show that the fractional Laplacian of any positive, non-integer order can be represented as a higher-order Neumann derivative of an extended function $U$, where $U$ satisfies a higher-order elliptic equation. 

To illustrate the technique, we first show that in the case $1 < \gamma < 2$, the fractional Laplacian $(-\Delta)^\gamma$ can still be represented as a suitable Neumann derivative for the solution of a higher order equation, and subsequently we generalize this to all positive, non-integer values of $\gamma$. Chang and Gonzalez \cite{C-G} showed that the extension has an interesting interpretation in terms of scattering theory, and, in particular that the following equation
\begin{equation}\label{keyequation}
 \Delta U + \frac{a}{y} U_y = 0
\end{equation}
(where $a = 1-2\gamma$) holds for the extended function $U$, and still holds for higher orders, so long as $\gamma \leq \frac{n}{2}$ (for $n$ odd, and all higher orders for $n$ even). The equation (\ref{keyequation}) plays an interesting role in the extension; specifically, we will see that it holds even when $\gamma \geq \frac{n}{2}$, and indeed does so for all noninteger $\gamma$. 

We conclude the paper with a new proof for strong unique continuation for fractional-harmonic functions of higher non-integer order. The proof uses both our extension technique to characterize fractional harmonic functions as boundary data for a very high order elliptic equation. We then use a technique by Garofalo and Lin \cite{Garofalo-Lin}, which uses a variation on the Almgren frequency formula to demonstrate unique continuation. 

The extension technique of Caffarelli and Silvestre has previously been adapted to analyze fractional-order powers of Schrodinger operators by Stinga and Torrea \cite{Stinga-Torrea}, although the overall order of the operator remained $\gamma < 1$. Fall and Felli \cite{Fall-Felli} have recently used the Almgren frequency formula to analyze unique continuation for fractional-order elliptic equations, although their analysis is also restricted to $\gamma < 1$. 

\section{The model case: $\gamma = \frac{3}{2}$}
\subsection{}
First we discuss the extension in a special case, which illustrates the main point of the argument without the complexity of notation we need for more general cases. In what follows, $\gamma =\frac{3}{2}$, and $a = 1-2\gamma = -2$. 

\begin{thm}\label{firsttheorem}
For functions $U \in W^{2,2}(\mathbb{R}^{n+1}_+)$ satisfying the equation
\begin{equation}
 \Delta^2 U(x,y) = 0 
\end{equation}
on the upper half space for $(x,y) \in \mathbb{R}^n \times \mathbb{R}_+$, where $y$ is the special direction, 
and the boundary conditions
\begin{eqnarray*}
U(x,0) &=& f(x) \\
U_{y}(x,0) &=& 0
\end{eqnarray*}
along $\{y=0\}$ where $u(x)$ is some function defined on $H^{\gamma}(\mathbb{R}^n)$ we have the result that
\begin{equation}
(-\Delta)^\frac{3}{2} f(x) = C_{n,\gamma} \frac{\partial}{\partial y} \Delta U(x,0)
\end{equation}
Specifically, 
\begin{equation}
 \int_{\mathbb{R}^n} |\xi|^{3} |\hat{f}(\xi)|^2 d\xi = C_{n,\gamma} \int_{\mathbb{R}^{n+1}_+} |\Delta U(x,y)|^2 dx dy 
\end{equation}
\end{thm}

\begin{proof}
Taking the Fourier transform in the $x$ variable only on the energy term $\Delta u$ we get 
\[ -|\xi|^2 \hat{U}(\xi,y) + \hat{U}_{yy}(\xi,y) \]
And so minimizing the energy corresponds to minimizing the integral 
\[ \int_{\mathbb{R}_+} \int_{\mathbb{R}^n} |-|\xi|^2 \hat{U}(\xi,y) + \hat{U}_{yy}(\xi,y)|^2 d\xi dy. \]
Integrating by parts, we see that for each value of $\xi$, $\hat{U}$ solves the ODE
\[ |\xi|^4 \hat{U} - 2|\xi|^2 \hat{U}_{yy} + \hat{U}_{yyyy} = 0. \]

Let $\phi \in W^{2,2}(\mathbb{R}_+)$ be the minimizer of the functional 
\[ J(\phi) = \int_{\mathbb{R}_+} (\phi''(y) - \phi(y))^2 dy \]
among functions satisfying the conditions $\phi(0) = 1$, $\phi'(0) = 0$. Thus $\phi$ solves the ODE
\[ \phi(y) - 2 \phi''(y) + \phi''''(y) = 0 \]
with appropriate boundary conditions and we see that $\hat{U}(\xi,y) = \hat{f}(\xi) \phi(|\xi|y)$ is a good representation for $\hat{U}$. 

Calculating, we see that 
\begin{eqnarray*}
\int (\Delta U)^2 dx dy &=& C_n \int \left|-|\xi|^2 \hat{U} + \hat{U}_{yy}\right|^2 d\xi dy \\
&=& C_n \int \left|-|\xi|^2 \hat{f}(\xi) \phi(|\xi|y) + |\xi|^2 \hat{f}(\xi) \phi''(|\xi|y)\right|^2 d\xi dy \\
&=& C_n \int |\xi|^4 |\hat{f}(\xi)|^2 (-\phi(\bar{y}) + \phi''(\bar{y}))^2  \frac{d\bar{y}}{|\xi|} d\xi \\
&=& C_n J(\phi) \int |\xi|^3 |\hat{f}(\xi)|^2 d\xi
\end{eqnarray*}
and hence the energies are identical up to a constant. 

The Euler-Lagrange equation for the left hand side above is simply the bi-Laplace equation, while for the right hand side it is the fractional harmonic equation of order $\gamma$, and the rest of the result follows.
\end{proof}

\subsection{}\label{sec:regularized-energy}
We remark that, using the equation 
\[ \Delta U - \frac{2 U_y}{y} = 0 \]
and the boundary conditions, we can perform a brief calculation
\begin{eqnarray*}
\int_{y > \epsilon} (\Delta U)^2 dx dy &=& - \int_{y=\epsilon} (\Delta U) \frac{\partial U}{\partial y} dx - \int_{y>\epsilon} \nabla (\Delta U) \cdot \nabla U dx dy \\
&=&   \int_{y=\epsilon} (\Delta U) \frac{\partial U}{\partial y} dx- 2 \int_{y>\epsilon} \nabla \left(\frac{U_y}{y}\right) \cdot \nabla U dx dy \\
&=&  \int_{y=\epsilon} (\Delta U) \frac{\partial U}{\partial y} dx + 2 \int_{y>\epsilon} \frac{U_y^2}{y^2} - \frac{\nabla U \cdot \nabla U_y}{y} dx dy \\
&=&  \int_{y=\epsilon} (\Delta U) \frac{\partial U}{\partial y} dx + 2 \int_{y=\epsilon} \frac{1}{2} \frac{|\nabla U|^2}{2y} dx + \int_{y>\epsilon} 2 \frac{U_y^2}{y^2} - \frac{|\nabla U|^2}{y^2} dx dy 
\end{eqnarray*}
We take the limit as $\epsilon \rightarrow 0$ and use that $U_y(x,0) = 0$, to recover 
\[ \frac{1}{2} \int_{\mathbb{R}^{n+1}_+} (\Delta U)^2 dx dy = \lim_{\epsilon \rightarrow 0} \left( \frac{1}{\epsilon} \int_{\mathbb{R}^n} |\nabla_x f|^2 dx - \int_{\mathbb{R}^{n+1}_+} \frac{|\nabla U|^2}{y^2} dx dy\right) \]
for solutions $U$ of our extension problem. This relationship gives a relationship between the new energy for the extension in this case, and the straightforward extension of the old energy, $\int y^{1-2\gamma} |\nabla U|^2$, which is infinite in this case. If we take the point of view of considering the extension problem as a version of the scattering problem a la \cite{C-G}, we have recovered our energy, $\int (\Delta U)^2 dx dy$, as the equivalent of the finite part, of the energy integral  $\int y^{-2} |\nabla U|^2$ (see, e.g., \cite[Proposition 3.2]{G-Z}). A similar calculation is possible in the general of non-integer $\gamma$, yielding many more boundary terms which blow up to infinity at different rates.

\section{$1< \gamma < 2$}

In this case, the argument is precisely analogous to the previous section, except that, like Caffarelli and Silvestre, we shall use a weighted seminorm. To be precise, we attach the weighted measure $y^b dy dx$ to our Sobolev spaces, and consider energy minimizers with respect to this measure on the upper half space of an appropriate energy. Here, we take $b= 3-2\gamma$. 

To construct the appropriate energy in this space, we introduce the following operator, which is a variant of the Laplacian adapted to the measure, whose virtue is that in the weighted space it behaves under integration by parts just as the regular Laplacian does in an unweighted space. 
\[ \Delta_b U = \Delta U + \frac{b}{y} U_y \]
gives us the desired relationship:
\[ \int_{\mathbb{R}^{n+1}_+} (\nabla \Phi \cdot \nabla \Psi) y^b dy dx = -\int_{\mathbb{R}^n} \Phi \lim_{y\rightarrow 0} \left(y^b \frac{\partial \Psi}{\partial y}\right) dx - \int_{\mathbb{R}^{n+1}_+} \Phi (\Delta_b \Psi) y^b dy dx \]
Clearly, the appropriate 2nd order seminorm for our space is 
\[ \int_{\mathbb{R}^{n+1}_+} y^b |\Delta_b U|^2 dy dx \]
Our space will be equipped with the norm
\[ \|U\|_{W^{2,2}(\mathbb{R}^{n+1}_+,y^b)}^2 = \| y^\frac{b}{2} \Delta_b U\|_{L^2(\mathbb{R}^{n+1}_+)}^2 + \| y^\frac{b}{2} \nabla U\|_{L^2(\mathbb{R}^{n+1}_+)}^2 + \|y^\frac{b}{2} U \|_{L^2(\mathbb{R}^{n+1}_+)}^2 \]

Our main result is now that
\begin{thm}\label{general-theorem}
For functions $U \in W^{2,2}(\mathbb{R}^{n+1}_+, y^b)$ satisfying the equation
\begin{equation}
 \Delta_b^2 U(x,y) = 0 
\end{equation}
on the upper half space for $(x,y) \in \mathbb{R}^n \times \mathbb{R}_+$, where $y$ is the special direction, 
and the boundary conditions
\begin{eqnarray*}
U(x,0) &=& f(x) \\
\lim_{y \rightarrow 0} y^b U_{y}(x,0) &=& 0
\end{eqnarray*}
along $\{y=0\}$ where $f(x)$ is some function defined on $H^{\gamma}(\mathbb{R}^n)$ we have the result that
\begin{equation}
(-\Delta)^\gamma f(x) = C_{n,\gamma} \lim_{y \rightarrow 0} y^b \frac{\partial}{\partial y} \Delta_b U(x,y)
\end{equation}
Specifically, 
\begin{equation}
 \int_{\mathbb{R}^n} |\xi|^{2\gamma} |\hat{f}(\xi)|^2 d\xi = C_{n,\gamma} \int_{\mathbb{R}^{n+1}_+} y^b |\Delta_b U(x,y)|^2 dx dy 
\end{equation}
\end{thm} 
\begin{proof}
Existence and uniqueness of a solution is guaranteed by the usual considerations. 
Taking the Fourier transform in the $x$ variable only on the equation $\Delta_b^2 u = 0$ we get 
\[ |\xi|^4 \hat{U} - (\frac{2b}{y}|\xi|^2 + \frac{b(b-2)}{y^3}) \hat{U}_y + (-2|\xi|^2 + \frac{b(b-1)}{y^2}) \hat{U}_{yy} + \frac{2b}{y} \hat{U}_{yyy} + \hat{U}_{yyyy}  = 0 \]
which is a 4th order ODE in $y$ for each value of $\xi$. Let $\phi \in W^{2,2}(\mathbb{R}_+, y^b)$ be the minimizer of the functional 
\[ J(\phi) = \int_{\mathbb{R}_+} y^b (\phi''(y) + \frac{b}{y} \phi'(y) - \phi(y))^2 dy \]
among functions satisfying the conditions $\phi(0) = 1$, $\phi'(0) = 0$. Thus $\phi$ solves the ODE
\[ \phi - (\frac{2b}{y} + \frac{b(b-2)}{y^3}) \phi' + (-2 + \frac{b(b-1)}{y^2}) \phi'' + \frac{2b}{y} \phi''' + \phi''''  = 0 \]
with appropriate boundary conditions and we see that $\hat{U}(\xi,y) = \hat{f}(\xi) \phi(|\xi|y)$ is a good representation for $\hat{U}$. 

Calculating, we see that 
\begin{eqnarray*}
\int y^b (\Delta_b U)^2 dx dy &=& C_n \int \left|-|\xi|^2 \hat{U} + \frac{b}{y} \hat{U}_y + \hat{U}_{yy}\right|^2 d\xi y^b dy \\
&=& C_n \int \left|-|\xi|^2 \hat{f}(\xi) \phi(|\xi|y) + \frac{b|\xi}{y} \hat{f}(\xi) \phi'(|\xi|y) + |\xi|^2 \hat{f}(\xi) \phi''(|\xi|y)\right|^2 d\xi y^b dy \\
&=& C_n \int |\xi|^4 |\hat{f}(\xi)|^2 (-\phi(\bar{y}) - \frac{b}{\bar{y}} \phi'(\bar{y}) +  \phi''(\bar{y}))^2  \frac{\bar{y}^b d\bar{y}}{|\xi|^{b+1}} d\xi \\
&=& C_n J(\phi) \int |\xi|^{2\gamma} |\hat{f}(\xi)|^2 d\xi
\end{eqnarray*}
and hence the energies are identical up to a constant. 

The Euler-Lagrange equation for the left hand side above is simply the modified bi-Laplace equation, while for the right hand side it is the fractional harmonic equation of order $\gamma$, and the rest of the result follows.

\end{proof}

\noindent
\textit{Remark}: It is not surprising (and easily verified) that solutions to $\Delta U + \frac{a}{y} U_y = 0$, with appropriate Dirichlet conditions, satisfy the equation $\Delta^2_b U = 0$. From the scattering theory (see, e.g., \cite{C-G}), we know that $\lim_{y\rightarrow 0} y^b U_y = 0$ for this equation, and thus from the uniqueness of solutions to the modified biharmonic equation we can see that the function of the extension described here, and the one in \cite{C-G}, are identical. 

\section{The General Case}
\subsection{}
The general case follows on a similar theme, taking progressively higher powers of the weighted Laplacian $\Delta_b$. In setting our boundary conditions, we take our cue from \cite{C-G}, whence we learn that, when $\gamma < \frac{n}{2}$ the extension function satisfies
\[ \Delta U + \frac{1-2\gamma}{y} U_y = 0 \]
and furthermore that, if $m < \gamma < m+1$, $U$ has a series expansion in $y$ that has only even integer powers, up to the power $y^{2\gamma}$. Being inspired by the belief that these considerations should hold for all non-integer values of $\gamma$, we set boundary conditions for the $y$-derivatives of $U$ in the following result. 

\begin{thm}\label{thm:super-theorem}
Let $\gamma > 0$ be some non-integer, positive power of the Laplacian. Let $m < \gamma < m+1$, or $m = [\gamma]$, and $b(\gamma) = 2m + 1 - 2\gamma$. 
For functions $U \in W^{m+1,2}(\mathbb{R}^{n+1}_+, y^b)$ satisfying the equation
\begin{equation}
 \Delta_b^{m+1} U(x,y) = 0 
\end{equation}
on the upper half space for $(x,y) \in \mathbb{R}^n \times \mathbb{R}_+$, where $y$ is the special direction, 
and the boundary conditions are that $U(x,0) = u(x)$ along $\{y=0\}$, and, furthermore, that for every positive odd integer $2k+1 < m+1$, we have
$\lim_{y\rightarrow 0} y^b \frac{\partial^{2k+1} U}{\partial y^{2k+1}}(x,0) = 0$, where $f(x)$ is some function defined on $H^{\gamma}(\mathbb{R}^n)$. For even integers, we specify the relationship 
\[\frac{\partial^{2k} U}{\partial y^{2k}}(x,0) = (\Delta_x^k U(x,0)) \prod_{j=1}^k  \frac{1}{2\gamma - 4(j-1)} \]
Then we have the result that
\begin{equation}
(-\Delta)^\gamma f(x) = C_{n,\gamma} \lim_{y \rightarrow 0} y^b \frac{\partial}{\partial y} \Delta_b^{m} U(x,y)
\end{equation}
Specifically, if $m$ is odd, 
\begin{equation}
 \int_{\mathbb{R}^n} |\xi|^{2\gamma} |\hat{f}(\xi)|^2 d\xi = C_{n,\gamma} \int_{\mathbb{R}^{n+1}_+} y^b |\Delta_b^\frac{m+1}{2} U(x,y)|^2 dx dy 
\end{equation}
and if $m$ is even,
\begin{equation} 
\int_{\mathbb{R}^n} |\xi|^{2\gamma} |\hat{f}(\xi)|^2 d\xi = C_{n,\gamma} \int_{\mathbb{R}^{n+1}_+} y^b |\nabla \Delta_b^{\lfloor\frac{m+1}{2}\rfloor} U(x,y)|^2 dx dy
\end{equation} 
\end{thm}

\begin{proof}
The proof follows along the same lines as before, although some of the auxiliary calculations can be a bit complicated, and are safely left to a lemma that we prove later. As in the previous cases, we consider a representation for the solution given by $\hat{U}(\xi,y) = \hat{f}(\xi) \phi(|\xi|y)$. If $m$ is an odd integer, then we let $\phi$ be the minimizer of 
\[ J(\phi) = \int_{\mathbb{R}^+} y^b \left(\left( -1 + \frac{b}{y}\frac{\partial}{\partial y} + \frac{\partial^2}{\partial^2 y} \right)^{\frac{m+1}{2}} \phi\right)^2 dy \]
subject to $\phi^{(2k+1)}(0) = 0$ for all odd $2k+1 < 2m+1$, and $\phi(0) = 1$, with $\phi^{(2k)}(0) = \prod_{j=1}^k \frac{1}{2\gamma - 4(j-1)}$. 
If $m$ is an even integer, then we let 
\[W[\phi,y] = \left( -1 + \frac{b}{y}\frac{\partial}{\partial y} + \frac{\partial^2}{\partial^2 y} \right)^{\lfloor \frac{m+1}{2} \rfloor} \phi\]
and take
\[ J(\phi) = \int_{\mathbb{R}^+} y^b \left( \left( W[\phi,y] \right)^2  + \left(\frac{\partial}{\partial y} W[\phi,y] \right)^2 \right ) dy \]
The equivalence of the energies then follows precisely as in the argument in Theorem \ref{general-theorem}. 

Only the calculation of the Euler-Lagrange equation for our energy functional remains. It is readily seen that $\Delta^{m+1}_b U = 0$, but it is necessary to show that boundary terms of the form $\int_{\mathbb{R}^{n}} \lim_{y \rightarrow 0} y^b \frac{\partial(\Delta^k_b U(x,y))}{\partial y} \Delta^{m-k}_b U(x,y) dx = 0$. In fact, $\frac{\partial(\Delta^k_b U(x,0))}{\partial y} = 0$. The calculation is technical, and we leave it to Lemma \ref{lem:tricky}. 
\end{proof}

\begin{lem}\label{lem:induction}
Functions $U$ satisfying the equation 
\[ \Delta U + \frac{1-2\gamma}{y} U_y = 0 \]
also satisfy
\[ \Delta^{m+1}_b U = 0\]
and $W = \Delta^k U$ satisfies
\[ \Delta W + \frac{2k+1-2\gamma}{y} W_y = 0 \]
for any $k < m +1$. Here $m,b$ are defined in terms of $\gamma$ as in the statement of Theorem \ref{thm:super-theorem}. 
\end{lem}
\begin{proof}
The proof shall be an induction on $m$. When $m=0$, this is a restatement of the definition of $\Delta_b$. Notationally, let us write $a = 1-2\gamma$. 

Suppose now, for general $m$, that $\Delta U + \frac{a}{y} U_y = 0$. Let $W= \Delta_b U$. Then 
\begin{eqnarray*}
 \Delta W &=& \Delta \left( \frac{b-a}{y} U_y \right) \\
&=& \frac{2(b-a)}{y^3} U_y - \frac{2(b-a)}{y^2} U_{yy} + \frac{b-a}{y} (\Delta U_y) \\
&=& (2+a) \frac{(b-a)}{y^2} \left( \frac{U_y}{y} - U_{yy} \right) = -\frac{(a+2)}{y} W_y
\end{eqnarray*}
and by the inductive hypothesis the result follows. 
\end{proof}

\noindent\textit{Remark} It is no surprise that, having carefully set the boundary conditions to coincide with the function $U$ from scattering theory \cite{C-G}, that our energy minimizer, satisfying the same equation as the $U$ of the scattering theory, would be exactly the same function by the uniqueness of energy minimizers. 

We shall now need a partial converse to the last lemma. In the case of $\gamma < \frac{n}{2}$, it is already known, and a consequence of the scattering result. 
\begin{lem}\label{lem:induction-converse}
Solutions of the equation $\Delta^{m+1}_b U = 0$ with boundary conditions given in Theorem \ref{thm:super-theorem} also satisfy the equation
\begin{equation}\label{myequation}
\Delta U + \frac{1-2\gamma}{y} U_y = 0  
\end{equation}
\end{lem}
\begin{proof}
We may assume $\gamma > 1$, because otherwise our statement is that of the original Caffarelli-Silvestre extension. 

In any subdomain strictly bounded away from the line $y = 0$, the equation (\ref{myequation}) is a classical 2nd order equation with bounded and regular coefficients, and hence we can solve the equation with boundary conditions given by the energy minimizer $U$; let's call this new solution $U'$. A result of Lin and Wang \cite{L-W} gives $C^{1,\alpha}$ up to the boundary regularity for solutions of the equation (\ref{myequation}). Hence we may approach by taking domains that approach ever more closely to the boundary, reapplying our argument, and passing to the limit. In the limit, it is no surprise that the limiting value of $U'$ is the same as the original $U$, since the two both satisfy $(\Delta_b)^{m+1} U = 0$ with the same boundary conditions. 
\end{proof}

We now have the tools to complete the proof:
\begin{lem}\label{lem:tricky}
For minimizers $U$ of our energy satisfying the preconditions of Theorem \ref{thm:super-theorem}, we have that 
\[ \frac{\partial(\Delta^k_b U(x,0))}{\partial y} = 0 \]
for $k < m$. 
\end{lem}
\begin{proof}
Let $W = \Delta^k_b U$. Then $W$ solves $\Delta^{m+1-k}_b W = 0$, and hence satisfies the equation 
\[ \Delta W + \frac{a+2k}{y} W_y = 0 \]
By the maximum principle, $W(x,0)+Cy^2$ is locally a supersolution for $C > 0$ sufficiently large, and $W(x,0) - Cy^2$ is a subsolution, again for $C>0$ sufficiently large. Hence, we can conclude that $\frac{\partial W}{\partial y} (x,0) = 0$. 
\end{proof}

As a consequence of our results, we have the following trace inequality:
\begin{cor}
For $U \in H^{m+1}(y^b, \mathbb{R}^{n+1}_+)$ satisfying the given boundary conditions, we have that the trace of $U$ on the boundary, which we call $f$, satisfies
\[ \|f\|_{H^\gamma(\mathbb{R}^n)} \leq C_\gamma \|y^b D^{m+1} U\|_{L^2(\mathbb{R}^{n+1}_+)} \]
where $D^{m+1}$ is either $\Delta_b^{\frac{m+1}{2}}$, or $\nabla \Delta_b^\frac{m}{2}$, whichever is appropriate. 
\end{cor}
The proof follows from the energy argument of Theorem \ref{thm:super-theorem}, since that theorem uses energy minimizers. 

\section{Strong unique continuation}\label{sec:monotonicity}
We show an analogue of Almgren's frequency formula holds for the fractional Laplacian, which allows us to deduce strong unique continuation in the style of Garofalo and Lin \cite{Garofalo-Lin}. 

We begin by showing that such a formula holds for balls centered on the interior of a domain where the equation $\Delta_b^{m+1} U = 0$ holds. The calculations along the boundary are not so very different, and the boundary terms are absorbed 
Let $U_k = (\Delta_b)^k U$, and we consider the behavior of certain integrals of $U$ with respect to balls centered at a particular point. Let
\[ D(r) = \sum_{k=0}^m \int_{B_r} y^b \left(|\nabla U_k|^2  + U_k U_{k+1}\right) dx dy\]
and
\[ H(r) = \sum_{k=0}^m \int_{\partial B_r} y^b (U_k)^2 dS. \]
Then we let 
\[ N(r) = r \frac{D(r)}{H(r)} \]
and we get the following result
\begin{thm}\label{thm:monotonicity-interior}
There exists a number $\Lambda(m,b,n) > 0$ such that if $(\Delta_b)^{m+1} U = 0$, then $e^{\Lambda r} N(r)$ is uniformly bounded for $r<1$. 
\end{thm}

In the proof, an important calculational tool is the following modification of an identity originally due to Rellich \cite{Rellich}, which appeared in \cite{C-S}. 
\begin{lem}\label{lem:rellich}
Suppose $\Delta_b W(x,y) = V(x,y)$. Then 
\[ r \int_{\partial B_r} y^b (|\nabla W|^2 - 2 W_r^2) dS = (n-1+b) \int_{B_r} y^b (\nabla W)^2 dx dy - 2 \int_{B_r} y^b (X \cdot \nabla W) V dx dy \]
\end{lem}  
\begin{proof}[Proof of Lemma \ref{lem:rellich}]
We apply the divergence theorem to the following term.
\[ \nabla \cdot \left( y^b \frac{|\nabla W|^2}{2} X - y^b (X \cdot \nabla W) \nabla W\right)\]
The boundary term is precisely the desired left hand side, and after some calculation we find that the right hand side matches up with the given term as well. 
\end{proof}

Another extremely useful estimate is the following calculation, which allows us to bound the behavior of an integral in the interior of the ball in terms of its behavior on the boundary. 
\begin{lem}\label{lem:interior-exterior}
Suppose $\Delta_b W(x,y) = V(x,y)$. Then 
\[ \int_{B_r} y^b W^2 dx dy \leq C \left(r \int_{\partial B_r} y^b W^2 dS + \int_{B_r} V^2 dx dy \right) \]
\end{lem}
\begin{proof}[Proof of Lemma \ref{lem:interior-exterior}]
Notice that $\Delta_b W^2 = 2|\nabla W|^2 + 2W \Delta_b W = 2|\nabla W|^2 + 2W V$. 
We consider the integral 
\begin{eqnarray*}
 \int_{B_r} y^b \Delta_b (W^2) \left(r^2 - |X|^2\right) dx dy &=& \int_{\partial B_r} y^b \nabla (W^2) \cdot \nu (r^2 - |X|^2) dS - \int_{B_r} y^b 2 W \nabla W \cdot (- 2 X) dx dy \\
&=& 4 \int_{B_r} y^b W (\nabla W \cdot X) dx dy \\
&=& 4 \int_{\partial B_r} y^b W^2 r dS - \int_{B_r} 4 W \nabla \cdot \left(y^b W X\right) dx dy \\
&=& 4 r \int_{\partial B_r} y^b W^2 dS - 4 \int_{B_r} y^b W \left[ (n+1 + b) W +  (\nabla W\cdot X) \right] dx dy \\
\end{eqnarray*}
Taking the last line of this calculation and comparing it with the second, we see that
\[ \int_{B_r} y^b W (\nabla W \cdot X) dx dy = \frac{1}{2} \left[ r \int_{\partial B_r} y^b W^2 dS - (n+1+b) \int_{B_r} y^b W^2 dx dy \right] \]
and hence that 
\[ \int_{B_r} y^b |\nabla W|^2 dx dy + (n+1+b) \int_{B_r} y^b W^2 dx dy  = r\int_{\partial B_r} y^b W^2 dS - \int_{B_r} y^b W V dx dy \]
Applying Cauchy's inequality, we get the desired result. 
\end{proof}

With these two estimates in hand, we can proceed to the proof. The general idea of the proof is the following: the map $r \mapsto N(r)$ is continuous in $r$, and so we can divide it into the set where $N(r)$ is large, and where it is small. We are only interested in proving our estimate when $N(r)$ is large, for the rest of the time it is already known to be small. 

\begin{proof}[Proof of Theorem \ref{thm:monotonicity-interior}]
We confine ourselves to the study of the set where $N(r) > 1$. 

We consider the expression $\log N(r)$, and take its derivative, whence we get
\[ \frac{1}{r} + \frac{D'(r)}{D(r)} - \frac{H'(r)}{H(r)} \]
Calculating and using the lemmata above, we find that 
\[ D'(r) = \left[\sum \int_{\partial B_r} y^b \left((\nabla U_k)^2 + U_k U_{k+1}\right) dS\right] \]
We examine first the gradient square term. 
\begin{eqnarray*}
\int_{\partial B_r} y^b (\nabla U_k)^2 dS &=& \frac{n-1+b}{r} \int_{B_r} y^b (\nabla U_k)^2 dx dy + 2 \int_{\partial B_r} y^b \left|\frac{\partial U_k}{\partial r}\right|^2 dS \\
&&- \frac{2}{r} \int_{B_r} y^b (X \cdot \nabla U_k) U_{k+1} dx dy\\
&=& \frac{n-1+b}{r} \int_{B_r} y^b (\nabla U_k)^2 dx dy  + 2 \int_{\partial B_r} y^b \left|\frac{\partial U_k}{\partial r}\right|^2 dS  \\
&&- 2 \int_{\partial B_r} y^b U_k U_{k+1} dS + 2\frac{n+1+b}{r} \int_{B_r} y^b U_k U_{k+1} dx dy\\ 
&& + \frac{2}{r} \int_{B_r} y^b U_{k+1} X \cdot \nabla U_k dx dy
\end{eqnarray*} 
Hence if we write
\[ D_k(r) = \int_{B_r} y^b \left(|\nabla U_k|^2 + U_k U_{k+1}\right) dy dx \]
then we have 
\begin{eqnarray*}
D_k'(r) &=& \frac{n-1+b}{r} D_k(r) + 2 \int_{\partial B_r} y^b \left|\frac{\partial U_k}{\partial r}\right|^2 dS - 2 \int_{\partial B_r} y^b U_k U_{k+1} dS\\
&& + \frac{n+3+b}{r} \int_{B_r} y^b U_k U_{k+1} dx dy + \frac{2}{r} \int_{B_r} y^b U_{k+1} (X \cdot \nabla U_k) dx dy 
\end{eqnarray*}

We now show that the last 3 terms of this expression are bounded in terms of $H(r)$. The first term is a straightforward application of elementary inequalities:
\[ \left|\int_{\partial B_r} y^b U_k U_{k+1} dS\right| \leq  \int_{\partial B_r} y^b \left(| U_k |^2  + |U_{k+1}|^2\right) dS \]
When we sum over all values of $k$, this is clearly bounded by $2 H(r)$. 

The second term is bounded in a similar fashion. First, an elementary inequality:
\[ \frac{1}{r} \left|\int_{B_r} y^b U_k U_{k+1} dx dy \right| \leq \frac{1}{r} \int_{B_r} y^b \left(|U_k|^2 + |U_{k+1}|^2\right) dx dy, \]
followed by iterated use of Lemma \ref{lem:interior-exterior}:
\[ \frac{1}{r} \int_{B_r} y^b U_k^2 dx dy \leq  C\left( \int_{\partial B_r} y^b U_k^2 dS + \int_{B_r} y^b U_{k+1}^2 dx dy \right) \leq \cdots \leq C H(r) \]

The third term we control in terms of $D_k(r)$ and $H(r)$.
\[ \frac{1}{r} \int_{B_r} y^b U_{k+1} (X \cdot \nabla U_k) dx dy \leq \frac{1}{r} \int_{B_r} y^b U_{k+1}^2 dx dy + \int_{B_r} y^b |\nabla U_k|^2 dx dy \]
Then by iterated use of Lemma \ref{lem:interior-exterior} we can bound $\frac{1}{r} \int_{B_r} y^b U_{k+1}^2 dx dy$ by $H(r)$, while the sum over $k$ of $\int_{B_r} y^b |\nabla U_k|^2 dx dy$ is clearly bounded by $D(r) + H(r)$. 

Hence, we conclude that 
\[ \frac{D'(r)}{D(r)} = \frac{(n-1+b)}{r} + 2 \frac{\sum_k \int_{\partial B_r} y^b \left| \frac{\partial U_k}{\partial r}\right|^2 dS}{D(r)} + 
\frac{1}{D(r)} \left( O(H(r)) + O(D(r)) \right) \]
Now, $O(D(r))/D(r)$ is bounded, and in the range we consider, where $r< 1$ and $N(r) > 1$, so is $\frac{H(r)}{D(r)}$. 

We now turn our examination to the study of $\frac{H'(r)}{H(r)}$. A straightforward calculation yields that 
\[ H'(r) = \frac{n+b}{r} H(r) + 2 \int_{\partial B_r} y^b \sum_k U_k \frac{\partial U_k}{\partial r} dS \]
We note that $D_k(r) = \int_{\partial B_r} y^b U_k \frac{\partial U_k}{\partial r} dS$ through a simple integration by parts. 

Hence, in the range where $r< 1$, $N(r) > 1$, we can estimate that
\begin{eqnarray*}
 \log (N(r))' &=& \frac{1}{r} + \frac{D'(r)}{D(r)} - \frac{H'(r)}{H(r)} \\
&=& O(1) +  2 \left( \frac{\sum_k \int_{\partial B_r} y^b \left(\frac{\partial U_k}{\partial r}\right)^2 dS}{D(r) } - \frac{D(r)}{ \sum_k \int_{\partial B_r} y^b (U_k)^2 dS }\right)
\end{eqnarray*}
The second term is non-negative by Holder's inequality, and hence
\[ \log(N(r))' \geq -C \]
for some suitable constant $C$, which establishes the theorem. 
\end{proof}

Every step of the proofs contained above also apply along the (suitably smooth) boundary of a domain where $U$ solves $\Delta_b^{m+1} U = 0$, provided the all the appropriate normal derivatives are 0, and hence, if we define, about points $(x,0) \in \mathbb{R}^{n+1}_+$,
\[ D(r) = \sum_{k=0}^m \int_{B_r \cap \{y>0\} } y^b\left( |\nabla U_k|^2  + U_k U_{k+1} \right) dx dy\]
and
\[ H(r) = \sum_{k=0}^m \int_{\partial B_r \cap \{y> 0\}} y^b (U_k)^2 dS \].
and
\[ N(r) = r \frac{D(r)}{H(r)} \]
we get the following analogous result: 
\begin{thm}\label{thm:monotonicity-fractional}
The result of Theorem \ref{thm:monotonicity-interior} still holds if we consider functions $U$ which solve $\Delta_b^{m+1} U = 0$ in the upper half space with $\frac{\partial}{\partial y} \Delta_b^k U = 0$ for $0 \leq k \leq m$.
\end{thm}
Following the proof of Garofalo and Lin \cite{Garofalo-Lin} (pp 256-257), that such an integral is bounded, along with our extension result, suffices to provide a proof of \textit{strong unique continuation}, that is, 
\begin{cor}
Let $f(x)$ be a function on $\mathbb{R}^n$ such that 
\[ (-\Delta)^\gamma f = 0 \]
in some smooth domain $\Omega \subset \mathbb{R}^n$, with $f(x)$ defined suitably on $\mathbb{R}^n \setminus \Omega$, so that $f \in H^\gamma(\mathbb{R}^n)$. Then if $f$ vanishes of infinite order at some point $x_0 \in \Omega$, that is to say for sufficiently small $r$, we have 
\[ \int_{B_r(x_0)} |f(x)| dx = O(r^k) \]
for every $k \in \mathbb{N}$, then $f$ is uniformly 0 in $\Omega$. 
\end{cor}

\section*{Funding}
This work was supported by a NSF Mathematical Sciences Postdoctoral Research Fellowship (DMS Award number 1103786).
\section*{Acknowledgements}

I am indebted to several people for their encouragement and advice in the preparation of this paper. Fang Hua Lin suggested that a suitable variation of Almgren's frequency formula would hold for this extension and suggested viewing it as an elliptic system, leading to the discussion in \S \ref{sec:monotonicity}. Qing Jie suggested the characterization of our new energy as the finite part of the divergent energy integral $\int y^a |\nabla U|^2 dx dy$, leading to the discussion in \S \ref{sec:regularized-energy}. Alice Chang asked the question that led to this work, and provided much moral support.

\bibliography{higher_extension}{}

\providecommand{\bysame}{\leavevmode\hbox to3em{\hrulefill}\thinspace}
\providecommand{\MR}{\relax\ifhmode\unskip\space\fi MR }
\providecommand{\MRhref}[2]{%
  \href{http://www.ams.org/mathscinet-getitem?mr=#1}{#2}
}
\providecommand{\href}[2]{#2}
\begin{thebibliography}{1}

\bibitem{C-S}
L.A. Caffarelli and L.~Silvestre, \emph{An extension problem related to the
  fractional {L}aplacian}, Comm. Partial Differential Equations \textbf{32}
  (2007), no.~7-9, 1245--1260.

\bibitem{C-G}
Sun-Yung~Alice Chang and Mar{\'{\i}}a del~Mar Gonz{\'a}lez, \emph{Fractional
  {L}aplacian in conformal geometry}, Adv. Math. \textbf{226} (2011), no.~2,
  1410--1432. \MR{2737789 (2012k:58057)}

\bibitem{Fall-Felli}
Mouhamed~Moustapha Fall and Veronica Felli, \emph{Unique continuation property
  and local asymptotics of solutions to fractional elliptic equations},
  Preprint, arXiv:1301.5119 (2013).

\bibitem{Garofalo-Lin}
Nicola Garofalo and Fang-Hua Lin, \emph{Monotonicity properties of variational
  integrals, {$A_p$} weights and unique continuation}, Indiana Univ. Math. J.
  \textbf{35} (1986), no.~2, 245--268. \MR{833393 (88b:35059)}

\bibitem{G-Z}
C.~Robin Graham and Maciej Zworski, \emph{Scattering matrix in conformal
  geometry}, Invent. Math. \textbf{152} (2003), no.~1, 89--118. \MR{1965361
  (2004c:58064)}

\bibitem{L-W}
Fang~Hua Lin and Lihe Wang, \emph{A class of fully nonlinear elliptic equations
  with singularity at the boundary}, J. Geom. Anal. \textbf{8} (1998), no.~4,
  583--598. \MR{1724206 (2001a:35068)}

\bibitem{Rellich}
Franz Rellich, \emph{Darstellung der {E}igenwerte von {$\Delta u+\lambda u=0$}
  durch ein {R}andintegral}, Math. Z. \textbf{46} (1940), 635--636. \MR{0002456
  (2,56d)}

\bibitem{Stinga-Torrea}
Pablo~Ra{\'u}l Stinga and Jos{\'e}~Luis Torrea, \emph{Extension problem and
  {H}arnack's inequality for some fractional operators}, Comm. Partial
  Differential Equations \textbf{35} (2010), no.~11, 2092--2122. \MR{2754080
  (2012c:35456)}

\end{thebibliography}
\bibliographystyle{amsplain}

\end{document}